\documentclass[10pt]{article}
\usepackage{amsmath,amsfonts,amssymb,amsthm,epsfig,epstopdf,url,array,cite}

\theoremstyle{plain}
\newtheorem{thm}{Theorem}[section]
\newtheorem{lem}[thm]{Lemma}

\newtheorem{cor}[thm]{Corollary}

\theoremstyle{definition}

\newtheorem{exmp}{Example}[section]

\theoremstyle{remark}

\newcommand{\ftnt}{\footnote}
\let\oldproofname=\proofname
\renewcommand{\proofname}
{\rm\bf{\oldproofname}}

\begin{document}
\title{\textrm{\textbf{Fixed point theorems for weak contraction in partially ordered $G$-metric space} }}

\author{\textbf{Snehasish Bose$^{1}$, Sk Monowar Hossein$^{2,*}$}}
\date{}
\maketitle
Department of Mathematics, Jadavpur University, Jadavpur-32, West Bengal, INDIA.
\ftnt{Email address: snehasishbose89@gmail.com}

Department of Mathematics, Aliah University, Sector-V, Kolkata-91, West Bengal, INDIA.
\ftnt{$^{,*}$ Corresponding author, email address: hossein@iucaa.ernet.in}

\begin{abstract}
  In this paper, we present some fixed point theorems in partially ordered G-metric space using the concept $(\psi,\phi)$- weak contraction which extend many existing fixed point theorems in such space. We also give some examples to show that if we transform a metric space into a G-metric space our results are not equivalent to the existing results in metric space.
\end{abstract}



\section{Introduction}
$~~$Let \textbf{X} be a nonempty set. A function $G:\textbf{X}\times\textbf{X}\times\textbf{X}\rightarrow[0,\infty)$ is called $G$-metric on $\textbf{X}$ if it satisfy the following properties :\begin{description}
                                               \item[(G1)] $G(x,y,z)=0$ if $x=y=z$,
                                               \item[(G2)] $0<G(x,x,y)$ for all $x,y\in~\textbf{X}$ with $y\neq~z$,
                                               \item[(G3)] $G(x,x,y)\leq~G(x,y,z)$ for all $x,y,z\in~\textbf{X}$ with $y\neq~z$,
                                               \item[(G4)] $G(x,y,z)=G(p\{x,y,z\})~~\forall~x,y,z\in\textbf{X}$, where $p$ is a permutation on \{x,y,z\},
                                               \item[(G5)] $G(x,y,z)\leq~G(x,a,a)+G(a,y,z)$ for all $x,y,z,a\in\textbf{X}$ (rectangle inequality).
                                             \end{description}
$~~$This notion of $G$-metric was introduced by Mustafa and Sims \cite{MS} in 2006. It can be shown that if $(\textbf{X},d)$ is a metric space one can define $G$-metric on $\textbf{X}$ by\\
$$G(x,y,z)=\textmd{max}\{d(x,y),d(y,z),d(z,x)\}~~~~\textmd{or}~~~~G(x,y,z)=d(x,y)+d(y,z)+d(z,x).$$
$~~$A self map $f$ on metric space $(\textbf{X},d)$ is said to be $\phi$-weak contraction if there exists a map $\phi:[0,\infty)\rightarrow[0,\infty)$ with $\phi(0)=0$ and $\phi(t)>0$ for all $t>0$ such that
$$d(fx,fy)\leq{d(x,y)-\phi(d(x,y))},~~~\forall~{x,y\in\textbf{X}}.\eqno{(1.1)}$$
$~~$A self map $f$ on metric space $(\textbf{X},d)$ is said to be $(\psi,\phi)$-weak contraction if there exists two maps $\psi,\phi:[0,\infty)\rightarrow[0,\infty)$ with $\psi(0)=\phi(0)=0$ and $\psi(t)>0$ and $\phi(t)>0$ for $t>0$ such that
$$\psi(d(fx,fy))\leq{\psi(d(x,y))-\phi(d(x,y))},~~~\forall~{x,y\in\textbf{X}}.\eqno{(1.2)}$$
$~~$A self map $f$ on metric space $(\textbf{X},d)$ is said to be generalized $(\psi,\phi)$-weak contraction if there exists two maps $\psi,\phi:[0,\infty)\rightarrow[0,\infty)$ with $\psi(0)=\phi(0)=0$ and $\psi(t)>0$ and $\phi(t)>0$ for $t>0$ such that
$$\psi(d(fx,fy))\leq{\psi(M(x,y))-\phi(M(x,y))},~~~\forall~{x,y\in\textbf{X}},\eqno{(1.3)}$$
where $M(x,y)$=max $\{d(x,y),d(x,fx),d(y,fy),\frac{1}{2}[d(x,fy)+d(y,fx)]\}$.\\
$~~$Using $\phi$-weak, $(\psi,\phi)$-weak and generalized $(\psi,\phi)$-weak contraction many authors studied existence of fixed points in complete metric spaces as well as parially ordered complete metric spaces. Some of them are Rhoades \cite{RHB}, Dutta and Choudhury \cite{DBS}, Dori\'{c} \cite{DDOR}, Popescu \cite{OP}, Moradia and Farajzadeh \cite{MA}, Harjani and Sadarangani \cite{HS}, Nashine and samet \cite{SANA}. Radenovi\'{c} and Kadelburg \cite{SRZK} showed that if $f$ is a self map on a complete partially ordered metric space $(\textbf{X},\preceq,d)$ with $x_0\preceq{fx_0}$ for some $x_0\in\textbf{X}$ and for any two comparable elements $x,y$ in $\textbf{X}$ there exists a continuous, non-decreasing function $\psi:[0,\infty)\rightarrow[0,\infty)$ and a lower semi-continuous function $\phi:[0,\infty)\rightarrow[0,\infty)$ satisfying $\psi(t)=\phi(t)=0$ iff $t=0$ such that
$$\psi(d(fx,fy))\leq{\psi(M(x,y))-\phi(M(x,y))},\eqno{(1.4)}$$
then in each of the following cases $f$ has a fixed point:
\begin{description}
  \item[\textit{(i)}] $f$ or $g$ is continuous, or
  \item[\textit{(ii)}] if a non-decreasing sequence $\{x_n\}$ converges to $x\in\textbf{X}$, then $x_n\preceq{x}$ for all n.
\end{description}
 $~$Existence of fixed point has important role in solving differential equations \cite{NILO,NILOO}, matrix equations \cite{RR} and integral equations. There are several works on fixed point in $G$-metric space \cite{MS1,MSB,JMB,SAV,KAG,RKA,RSPB}. But in 2012 Samet and Jeli \cite{JMB,SAV} showed that major amount of results were obtained by transforming the contraction condition in usual or quasi metric spaces context to G-metric spaces. Recently Karapinar and Agarwal \cite{KAG} proved that if $f$ is a self map on a $G$-metric space $\textbf{X}$ such that
$$G(fx,f^2x,fy)\leq{G(x,fx,y)-\phi(G(x,fx,y))},~~~\forall~x,y\in\textbf{X},\eqno{(1.5)}$$
where $\phi:[0,\infty)\rightarrow[0,\infty)$ is continuous function such that $\phi(t)=0$ iff $t=0$, then f has a unique fixed point. They also showed that the above contraction could not characterized in context of usual or quasi metric space as suggested in \cite{JMB,SAV}.\\
$~~$A partially ordered $G$-metric space is said to be regular non decreasing if for all $\preceq$-monotone non-decreasing sequence $\{x_n\}\in\textbf{X}$ such that $x_n\rightarrow~x^*$ implies $x_n\preceq~x^*$ for all $n\in\mathbb{N}$.\\
$~~$In this paper, using the concept of $(\psi,\phi)$-weak contraction we present some fixed point theorems in partially ordered $G$-metric space and we show that one of our result extend the fixed point theorem given by Karapinar and Agarwal \cite{KAG} on partially ordered $G$-metric space. We also give a sufficient condition for uniqueness of fixed point and an example to show that our result is not equivalent to the result of Radenovi\'{c} and Kadelburg \cite{SRZK}.

\section{Existence of fixed points}

Let us consider two sets
 $\Psi$ =\{$\psi:[0,\infty)\rightarrow[0,\infty):\psi$ is continuous, non-decreasing and $\psi(t)=0$ iff $t=0$\}\\
and $\Phi$ =\{$\phi:[0,\infty)\rightarrow[0,\infty):\phi$ is lower semi-continuous, and $\phi(t)=0$ iff $t=0$\}\\
\begin{thm}
\textit{Let $(\textbf{X},\preceq,G)$ be a complete partially ordered $G$-metric space. Let T : $\textbf{X}\rightarrow\textbf{X}$ be a mapping satisfying the following conditions:}
\begin{enumerate}
\item \textit{T is G-continuous or $(\textbf{X},\preceq,G)$ is regular non-decreasing,}\
\item \textit{T is non-decreasing,}\
\item \textit{There exists $x_0\in\textbf{X}$ with $x_0\preceq{Tx_0}$,}\
\item \textit{There exists $\psi\in\Psi$ and $\phi\in\Phi$ such that for all comparable $x\preceq{y}\preceq{z}$ in $\textbf{X}$}\
\end{enumerate}

$$\psi(G(Tx,Ty,Tz))\leq\psi(M(x,y,z))-\phi(M(x,y,z)),\eqno{(2.1)}$$
\textit{where M(x,y,z)= max $\{G(x,Tx,y),~G(x,Tx,z),~G(x,y,z),~G(y,Ty,Ty),\\
~~~~~~~~~~~~~~~~~~~~~~~~~~~~~~~~~~~~~~~~~G(z,Tz,Tz),~\frac{1}{2}[G(x,Ty,Tz)+G(Tx,y,z)]\}$. }\\
\textit{Then T has a fixed point.}
\end{thm}
\begin{proof}
 Let, $x_{n+1}=Tx_n$ for all $n=0,1,2,3,.......$.\\\\ Since T is non-decreasing, then $x_n\preceq~x_{n+1}$ for all $n\geq0$. So, from (2.1) we have,\\\\
$\psi(G(x_n,x_{n+1},x_{n+1}))=\\$
$$\psi(G(Tx_{n-1},Tx_n,Tx_n))\leq\psi(M(x_{n-1},x_n,x_n))-\phi(M(x_{n-1},x_n,x_n)),\eqno{(2.2)}$$
which implies that,
$$\psi(G(x_n,x_{n+1},x_{n+1}))\leq\psi(M(x_{n-1},x_n,x_n)).\eqno{(2.3)}$$
Since $\psi$ is monotone non-decreasing, we get
$$G(x_n,x_{n+1},x_{n+1}))\leq{M(x_{n-1},x_n,x_n)}.\eqno{(2.4)}$$
Now, by using the rectangle property of G, we have\\
$M(x_{n-1},x_n,x_n)~$= max $\{G(x_{n-1},Tx_{n-1},x_n),~G(x_{n-1},Tx_{n-1},x_n),~G(x_{n-1},x_n,x_n),~G(x_n,Tx_n,Tx_n),\\
~~~~~~~~~~~~~~~~~~~~~~~~~~~~~~~~~~~~~~~~~~~~~G(x_n,Tx_n,Tx_n),~\frac{1}{2}[G(x_{n-1},Tx_n,Tx_n)+G(Tx_{n-1},x_n,x_n)]\}$\\
$~~~~~~~~~~~~~~~~~~~~~~~=$ max $\{G(x_n,x_{n+1},x_{n+1}),~G(x_{n-1},x_n,x_n),~\frac{1}{2}G(x_{n-1},x_{n+1},x_{n+1})\}$\\
$~~~~~~~~~~~~~~~~~~~~~~~=$ max $\{G(x_n,x_{n+1},x_{n+1}),~G(x_{n-1},x_n,x_n)\}$\\
as, $\frac{1}{2}G(x_{n-1},x_{n+1},x_{n+1})\leq\frac{1}{2}[G(x_n,x_{n+1},x_{n+1})+G(x_{n-1},x_n,x_n)]$.\\\\
If $G(x_n,x_{n+1},x_{n+1})>G(x_{n-1},x_n,x_n)$, then\\
 $$M(x_{n-1},x_n,x_n)=G(x_n,x_{n+1},x_{n+1})>0\Longrightarrow\phi(G(x_n,x_{n+1},x_{n+1})>0,~~~~~~~~~~~~~~~~~~~~~~~~~~~~~~~~~~~~~~~~~~~~~$$ then from (2.2) we get \\
  $\psi(G(x_n,x_{n+1},x_{n+1}))\leq\psi(G(x_n,x_{n+1},x_{n+1}))-\phi(G(x_n,x_{n+1},x_{n+1})$,
which is a contradiction.\\ So, we have
  $$G(x_n,x_{n+1},x_{n+1})\leq~M(x_{n-1},x_n,x_n)=G(x_{n-1},x_n,x_n)$$
Hence, $\{G(x_{n-1},x_n,x_n)\}$ is a positive, non-decreasing sequence in $\mathbb{R}$, which is bounded below, so it is convergent. So there exists $a\geq0$ such that
  $$\lim_{n\rightarrow\infty}G(x_{n-1},x_n,x_n)=a.~~~~~~~~~~~~~~~~~~~~~~~~~~~~~~~~~~~~~~~~~~~~~~~~~~~~~~~~~~~~~~~~~~~~~~~~~~~~~~~~~~~~~~~~~~~~~~~~~~~~~~~~~~~~\eqno{(2.5)}$$
  Now if, $a>0$ then $\phi(a)>0$, and
  $$\textmd{since, $\phi$ is lower semi-continuous, }\phi(a)\leq{\liminf_{n\rightarrow\infty}G(x_{n-1},x_n,x_n)}~~~~~~~~~~~~~~~~~~~~~~~~~~~~~~~~~~~~~~~~~~~~~~~~~~~~~~~~~~~~~~~~~.$$
so letting $n\rightarrow\infty$ in (2.2), we get
  $$\psi(a)\leq\psi(a)-\phi(a),~~~~~~~~~~~~~~~~~~~~~~~~~~~~~~~~~~~~~~~~~~~~~~~~~~~~~~~~~~~~~~~~~~~~~~~~~~~~~~~~~~~~~~~~~~~~~~~~~~~~~~~~~~~~\eqno{(2.6)}$$
which is a contradiction. So $a=0$ that is
  $$\lim_{n\rightarrow\infty}G(x_{n-1},x_n,x_n)=0.~~~~~~~~~~~~~~~~~~~~~~~~~~~~~~~~~~~~~~~~~~~~~~~~~~~~~~~~~~~~~~~~~~~~~~~~~~~~~~~~~~~~~~~~~~~~~~~~~~~~~~~~~~~~\eqno{(2.7)}$$
  Now Since $G(x_{n-1},x_{n-1},x_n)\leq{2G(x_{n-1},x_n,x_n)}$,
  $$\lim_{n\rightarrow\infty}G(x_{n-1},x_{n-1},x_n)=0.~~~~~~~~~~~~~~~~~~~~~~~~~~~~~~~~~~~~~~~~~~~~~~~~~~~~~~~~~~~~~~~~~~~~~~~~~~~~~~~~~~~~~~~~~~~~~~~~~~~~~~~~~~~~\eqno{(2.8)}$$
Now, we show that $\{x_n\}$ is $G$-cauchy.\\
$~$Suppose that, $\{x_n\}$ is not $G$-cauchy. Then, there exist $\epsilon>0$ and subsequences $\{x_{n_k}\}$ and $\{x_{m_k}\}$ of $\{x_n\}$ with $n_k>m_k>k$ such that,
$$G(x_{m_k},x_{m_k},x_{n_k})\geq\epsilon~~~~~~~~~~~~~~~~~~~\forall~k\in~\mathbb{N}.~~~~~~~~~~~~~~~~~~~~~~~~~~~~~~~~~~~~~~~~~~~~~~~~~~~~~~~~~~\eqno{(2.9)}$$
Furthermore, corresponding to $m_k$, one can choose $n_k$ such that, it is the smallest integer with $n_k>m_k$ satisfying $(2.9)$. Then,
$$G(x_{m_k},x_{m_k},x_{n_k-1})<\epsilon~~~~~~~~~~~~~~~~~~\forall~k\in\mathbb{N}.~~~~~~~~~~~~~~~~~~~~~~~~~~~~~~~~~~~~~~~~~~~~~~~~~~~~~~~~~~\eqno{(2.10)}$$
So by using rectangle inequality and $(2.9)$, $(2.10)$ we get,
$$\epsilon\leq~G(x_{m_k},x_{m_k},x_{n_k})\leq{G(x_{m_k},x_{m_k},x_{n_k-1})+G(x_{n_k-1},x_{n_k-1},x_{n_k})}.\eqno{(2.11)}$$
Taking limit $k\rightarrow\infty$ in (2.11) we have
$$\lim_{n\rightarrow\infty}G(x_{m_k},x_{m_k},x_{n_k})=\epsilon.~~~~~~~~~~~~~~~~~~~~~~~~~~~~~~~~~~~~~~~~~~~~~~~~~~~~~~~~~~~~~~~~~~~~~~~~~~~~~~~~~~~~~~~~~~~~~~~~~~~~~~~~~~~~\eqno{(2.12)}$$
Again, $G(x_{m_k-1},x_{m_k-1},x_{n_k-1})\leq$
$$~~~~~~~~~~~~~~~~G(x_{m_k-1},x_{m_k-1},x_{m_k})+G(x_{m_k},x_{m_k},x_{n_k})+G(x_{n_k-1},x_{n_k},x_{n_k})\eqno{(2.13)}$$
and, $G(x_{m_k},x_{m_k},x_{n_k})\leq$
$$~~~~~~~~~~~~~~~~G(x_{m_k-1},x_{m_k},x_{m_k})+G(x_{m_k-1},x_{m_k-1},x_{n_k-1})+G(x_{n_k-1},x_{n_k-1},x_{n_k}).\eqno{(2.14)}$$
Letting $k\rightarrow\infty$ in $(2.13)$ and $(2.14)$,we get
$$\lim_{k\rightarrow\infty}G(x_{m_k-1},x_{m_k-1},x_{n_k-1})=\lim_{k\rightarrow\infty}G(x_{m_k},x_{m_k},x_{n_k})=\epsilon.~~~~~~~~~~~~~~~~~~~~~~~~~~~~~~~~~~~~~~~~~~~~~~~~~~\eqno{(2.15)}$$
Now, Since, $G(x_{m_k},x_{m_k},x_{n_k})=G(Tx_{m_k-1},Tx_{m_k-1},Tx_{n_k-1})$. So, by $(2.1)$
$$\psi(G(x_{m_k},x_{m_k},x_{n_k}))\leq\psi(M(x_{m_k-1},x_{m_k-1},x_{n_k-1}))-\phi(M(x_{m_k-1},x_{m_k-1},x_{n_k-1}))~~~~~~~~~~~~~~~~~\eqno{(2.16)}$$
Now, $M(x_{m_k-1},x_{m_k-1},x_{n_k-1})$\\  $~~~~~~~~~~$= max $\{G(x_{m_k-1},Tx_{m_k-1},x_{m_k-1}),~G(x_{m_k-1},Tx_{m_k-1},x_{n_k-1}),~G(x_{m_k-1},x_{m_k-1},x_{n_k-1}),\\
~~~~~~~~~~~~~~~~~~~~~G(x_{m_k-1},Tx_{m_k-1},Tx_{m_k-1}),~G(x_{n_k-1},Tx_{n_k-1},Tx_{n_k-1}),\\
~~~~~~~~~~~~~~~~~~\frac{1}{2}[G(x_{m_k-1},Tx_{m_k-1},Tx_{n_k-1})+G(Tx_{m_k-1},x_{m_k-1},x_{n_k-1})]\}$ \\\\
$~~~~~~~~~~$= max $\{G(x_{m_k-1},x_{m_k},x_{m_k-1}),~G(x_{m_k-1},x_{m_k},x_{n_k-1}),~G(x_{m_k-1},x_{m_k-1},x_{n_k-1}),\\
~~~~~~~~~~~~~~~~~~~~~G(x_{m_k-1},x_{m_k},x_{m_k}),~G(x_{n_k-1},x_{n_k},x_{n_k}),\\
~~~~~~~~~~~~~~~~~~~\frac{1}{2}[G(x_{m_k-1},x_{m_k},x_{n_k})+G(x_{m_k},x_{m_k-1},x_{n_k-1})]\}.~~~~~~~~~~~~~~~~~~~~~~~~~~~~~~~~~~~~~~~~~~~~~~
~~~(2.17)$ \\\\
Using Rectangle inequality, we get\\
$G(x_{m_k-1},x_{m_k},x_{n_k-1})\leq{G(x_{m_k-1},x_{m_k-1},x_{m_k})+G(x_{m_k-1},x_{m_k-1},x_{n_k-1})}$,\\
$G(x_{m_k-1},x_{m_k-1},x_{n_k-1})\leq{G(x_{m_k-1},x_{m_k},x_{m_k})+G(x_{m_k-1},x_{m_k},x_{n_k-1})}$,\\
that is, by using (2.8), (2.9) and (2.13), we have
$$\lim_{k\rightarrow\infty}G(x_{m_k-1},x_{m_k},x_{n_k-1})=\epsilon.~~~~~~~~~~~~~~~~~~~~~~~~~~~~~~~~~~~~~~~~~~~~~~~~~~~~~~~~~~~~~~~~~~~~~~~~~~~~~~~~~~~~~~~~~~~~~~~~~~~~~~~~~~~~$$
$$\textmd{Similarly,}~~~~\lim_{k\rightarrow\infty}G(x_{m_k-1},x_{m_k},x_{n_k})=\epsilon.~~~~~~~~~~~~~~~~~~~~~~~~~~~~~~~~~~~~~~~~~~~~~~~~~~~~~~~~~~~~~~~~~~~~~~~~~~~~~~~~~~~~~~~~~~~~~~~~~~~~~~~~~~~~$$
So, by using above inequalities, (2.7),(2.8),(2.15), we get
$$\lim_{k\rightarrow\infty}M(x_{m_k-1},x_{m_k-1},x_{n_k-1})=\epsilon~~~~~~~~~~~~~~~~~~~~~~~~~~~~~~~~~~~~~~~~~~~~~~~~~~~~~~~~~~~~~~~~~~~~~~~~~~~~~~~~~~~~~~~~~~~~~~~~~~~~~~~~~~~~\eqno{(2.18)}$$
therefore, letting $k\rightarrow\infty$ in $(2.16)$ and using $(2.18)$, we have
$$\psi(\epsilon)\leq\psi(\epsilon)-\phi(\epsilon).~~~~~~~~~~~~~~~~~~~~~~~~~~~~~~~~~~~~~~~~~~~~~~~~~~~~~~~~~~~~~~~~~~~~~~~~~~~~~~~~~~~~~~~~~~~~~~~~~~~~~~~~~~~~\eqno{(2.19)}$$
Since $\epsilon>0$, $(2.19)$ leads us to a contradiction.\\\\
Therefore, $\{x_n\}$ is a $G$-cauchy sequence. Since $(\textbf{X},G)$ is complete, there exists $x^*\in\textbf{X}$ such that $x_n\rightarrow~x^*$ as $n\rightarrow\infty$.\\
we claim that $x^*$ is the fixed point of $T$.\\
Case I: if $T$ are continuous, then
$$\lim_{n\rightarrow\infty}G(x_{n+1},x^*,x^*)=G(x^*,x^*,x^*)=0~~~~~~~~~~~~~~~~~~~~~~~~~~~~~~~~~~~~~~~~~~~~~~~~~~~~~~~~~~~~~~~~~~~~~~~~~~~~~~~~~~~~~~~~~~~~~~~~~~~~~~~~~~~~$$
$$\textmd{that is, }\lim_{n\rightarrow\infty}G(Tx_n,x^*,x^*)=G(Tx^*,x^*,x^*)=0.~~~~~~~~~~~~~~~~~~~~~~~~~~~~~~~~~~~~~~~~~~~~~~~~~~~~~~~~~~~~~~~~~~~~~~~~~~~~~~~~~~~~~~~~~~~~~~~~~~~~~~~~~~~~$$
So, $Tx^*=x^*$ that is $x^*$ is a fixed point of $T$.\\
Case II: if $(\textbf{X},\preceq,G)$ is regular non-decreasing, then $x_n\preceq{x^{*}}$. So,
 $$\psi(G(x_{n+1},Tx^*,Tx^*))=\psi(G(Tx_n,Tx^*,Tx^*))\leq\psi(M(x_n,x^*,x^*))-\phi(M(x_n,x^*,x^*)).\eqno{(2.20)}$$
Now, $M(x_n,x^*,x^*)$= max $\{G(x_n,Tx_n,x^*),~G(x_n,Tx_n,x^*),~G(x_n,x^*,x^*),~G(x^*,Tx^*,Tx^*),\\
~~~~~~~~~~~~~~~~~~~~~~~~~~~~~~~~~~~~~~~~~~~~~~~~~~~~~~~~G(x^*,Tx^*,Tx^*),~\frac{1}{2}[G(x_n,Tx^*,Tx^*)+G(Tx_n,x^*,x^*)]\}$ \\
$~~~~~~~~~~~~~~~~~~~~~~~~~~$= max $\{G(x_n,x_{n+1},x^*),~G(x_n,x^*,x^*),~G(x^*,Tx^*,Tx^*),\\
~~~~~~~~~~~~~~~~~~~~~~~~~~~~~~~~~~~~~~~~~~~~~~~~~~~~~~~\frac{1}{2}[G(x_n,Tx^*,Tx^*)+G(x_{n+1},x^*,x^*)]\}.~~~~~~~~~~~~~~~~~~~~~~~~~~~~(2.21)$

$$\textmd{Therefore,}~~~\lim_{n\rightarrow\infty}M(x_n,x^*,x^*)= G(x^*,Tx^*,Tx^*).~~~~~~~~~~~~~~~~~~~~~~~~~~~~~~~~~~~~~~~~~~~~~~~~~~~~~~~~~~~~~~~~~~~~~~~~~~~~~~~~~~~~~~~~~~\eqno{(2.22)}$$
If $G(x^*,Tx^*,Tx^*)>0$, letting $n\rightarrow\infty$ in $(2.20)$ give
$$\psi(G(x^*,Tx^*,Tx^*))\leq\psi(G(x^*,Tx^*,Tx^*))-\phi(G(x^*,Tx^*,Tx^*))$$
which is a contradiction. So, $G(x^*,Tx^*,Tx^*)=0\Longrightarrow~Tx^*=x^*$. Hence T has a fixed point.
\end{proof}

\begin{exmp}
\textbf{(Existence of fixed point in case of not continuous function)} Let $\textbf{X}=[0,1]$ and $x\preceq{y}$ iff $x\leq{y}$ and defined G-metric on \textbf{X} by $G(x,y,z)$=max$\{|x-y|,|y-z|,|z-x|\}$. Then $(\textbf{X},\preceq,G)$ is complete partially ordered G-metric space. Consider the mapping $T:\textbf{X}\rightarrow\textbf{X}$ by
$T(x)=0$ for $x\in[0,1)$ and T(1)=$\frac{1}{4}$. Then $G(Tx,Ty,Tz)\leq~\frac{3}{4}M(x,y,z)~~~~\forall~x\preceq{y}\preceq{z}$. So by taking $\psi(t)=t$ and $\phi(t)=\frac{1}{2}t$, and T satisfies all conditions in the previous theorem. Notice that T is not G-continuous and has a fixed point at 0.
\end{exmp}
By following examples we will show that if we violet any one of required conditions of the above theorem then $T$ may not have a fixed point.
\begin{exmp}
Let $\textbf{X}=[0,1]$ and let $x\preceq{y}$ iff either $x=y$ or $xy(x-y)>0$. Now defined G-metric on \textbf{X} by $G(x,y,z)$=max $\{|x-y|,|y-z|,|z-x|\}$. Then $(\textbf{X},\preceq,G)$ is complete partially ordered G-metric space and if $x\neq0$ then $x$ and 0 are not comparable. Consider the mapping $T:\textbf{X}\rightarrow\textbf{X}$ by\\ $T(x)=0$ for $x\in(0,1]$ and T(0)=$\frac{1}{4}$. Then $G(Tx,Ty,Tz)\leq~\frac{1}{2}M(x,y,z)~~~~\forall~x\preceq{y}\preceq{z}$. So by taking $\psi(t)=t$ and $\phi(t)=\frac{1}{4}t$, \textbf{T satisfies only  conditions 2, 3 and 4 of Theorem 2.1. Here, T has no fixed point in $\textbf{X}$}.
\end{exmp}
\begin{exmp}
Let $\textbf{X}=\{2,3,4\}$ and $x\preceq{y}$ iff $x\mid{y}$ and defined G-metric on \textbf{X} by $G(x,y,z)$=max$\{|x-y|,|y-z|,|z-x|\}$. Then $(\textbf{X},\preceq,G)$ is complete partially ordered G-metric space and 3 is not comparable with 2 and 4. Let us consider the mapping $T:\textbf{X}\rightarrow\textbf{X}$ by T(2)=T(3)=4  and T(4)=3. Then $G(Tx,Ty,Tz)\leq~\frac{1}{2}M(x,y,z)~~~~\forall~x\preceq{y}\preceq{z}$. So by taking $\psi(t)=t$ and $\phi(t)=\frac{1}{4}t$, \textbf{T satisfies only conditions 1, 3 and 4 of Theorem 2.1. Here, T has no fixed point in $\textbf{X}$}.
\end{exmp}
\begin{exmp}
Let $\textbf{X}=\{2,3\}$ and $x\preceq~y$ iff $x\mid{y}$ and defined G-metric by $G(x,y,z)$=max $\{|x-y|,|y-z|,|z-x|\}$. Then $(\textbf{X},\preceq,G)$ is complete partially ordered G-metric space and 2 is not comparable with 3. Consider the mapping $T:\textbf{X}\rightarrow\textbf{X}$ by T(2)=3 and T(3)=2. Then $G(Tx,Ty,Tz)=0~~~~\forall~x\preceq{y}\preceq{z}$. Therefore taking $\psi(t)=t$ and $\phi(t)=\frac{1}{2}t$, \textbf{T satisfies only conditions 1, 2 and 4 of Theorem 2.1. Here, T has no fixed point in $\textbf{X}$}.
\end{exmp}
\textbf{Next example shows that the same conclusion may not hold if $M(x,y,z)$ is replaced by}\\
$~M_{1}(x,y,z)$= max $\{G(x,Tx,y),~G(x,Tx,z),~G(x,y,z),~G(y,Ty,Ty),\\
~~~~~~~~~~~~~~~~~~~~~~~~~~~~~~~~~~~~~~~~~~~~~~~~~~~~~~~~G(z,Tz,Tz),~G(x,Ty,Tz),~G(Tx,y,z)\}$
\begin{exmp}Let $\textbf{X}=\{2^n~|~n\in\mathbb{N}\}$ and $x\preceq~y$ iff $x\mid{y}$. Now defined G-metric on \textbf{X} by $G(x,y,z)$= max $\{|x-y|,|y-z|,|z-x|\}$. Then $(\textbf{X},\preceq,G)$ is complete partially ordered G-metric space. Consider the mapping $T:\textbf{X}\rightarrow\textbf{X}$ by $T(2^n)=2^{n+1}~~~\forall~n\in\mathbb{N}$. Then $M_{1}(x,y,z)-G(Tx,Ty,Tz)\geq2~~\forall~x\preceq{y}\preceq{z}$. Therefore taking $\psi(t)=t$ and $\phi(t)=\left\{
\begin{array}{ll}
\frac{1}{2}t & \mbox{if } x\in[0,2]\\
1 & \mbox{if } x\in(2,\infty)
\end{array}
\right.$
 , T satisfies all conditions in the previous theorem, with $M(x,y,z)$ replaced by $M_{1}(x,y,z)$. Obviously the mapping T has no fixed point in $\textbf{X}$.
\end{exmp}
\begin{cor}
Let T satisfy the conditions of Theorem 2.1, except the contraction defined in condition 4 is replaced by the following: for all comparable $x\preceq{y}\preceq{z}$ in $\textbf{X}$ there
exists a positive Lebesque integrable function $\varphi$ on $\mathbb{R}$ such that $\int^\epsilon_0{\varphi}>0$ for each  $\epsilon>0$ and that
$$\int^{\psi(G(Tx,Ty,Tz))}_0{\varphi(t)}dt\leq{\int^{\psi(M(x,y,z))}_0{\varphi(t)}dt-\int^{\phi(M(x,y,z))}_0{\varphi(t)}dt}.$$
Then T has a fixed point.
\end{cor}
\begin{proof}
Consider the function,$~~~~~~~~~~\tau(x)=\int^x_0{\varphi(x)}dx$,\\
then the above contraction reduces to
$$(\tau\circ\psi)(G(Tx,Ty,Tz))\leq{(\tau\circ\psi)(M(x,y,z))-(\tau\circ\phi)(M(x,y,z))},$$
so taking $\tau\circ\psi=\psi_1$ and $\tau\circ\phi=\phi_1$ and using Theorem 2.1 we obtain proof.

\end{proof}
Before the next result we will prove following lemmas.
\begin{lem}
\textit{Let $f,g~:~X\rightarrow[0,\infty)$ be two functions such that $f(x)\leq~g(x)~\forall~x\in~X$ then for a $\psi\in\Psi$ and $\phi\in\Phi$ with $\psi(x)\geq\phi(x)$ there exists $\phi_1\in\Phi$ such that $\psi(f(x))-\phi(f(x))\leq\psi(g(x))-\phi_1(g(x))$.}
\end{lem}
\begin{proof}
Take $\alpha>0$. Let $\psi(\alpha)=\epsilon$.
  Then $\exists~a_1>0$ such that $$\sup_{x\in[0,\alpha]}(\psi(x)-\phi(x))\leq\epsilon-a_1,$$
  otherwise, if $\exists$ no $a_1>0$ then $$\sup_{x\in[0,\alpha]}(\psi(x)-\phi(x))=\epsilon$$and since $\psi$ is nondecreasing,  $x\rightarrow\alpha$ implies $\phi(x)\rightarrow0$, which is a contradiction.\\
  Similarly for each $\frac{\alpha}{n},~~~~\exists~a_n>0$ such that$$\sup_{x\in[0,\frac{\alpha}{n}]}(\psi(x)-\phi(x))\leq\psi(\frac{\alpha}{n})-a_n.$$
  Let$$b=\inf_{x\in(\alpha,2\alpha)}\phi(x),~~~~~\textmd{and}~~~~~~b_n=\inf_{x\in(\frac{\alpha}{n+1},\frac{\alpha}{n}]}\phi(x),~~c_n=\inf_{x\in[(n+1)\alpha,(n+2)\alpha)}\phi(x)~~~~~~~~\forall~n\in\mathbb{N}$$
  Now if $b_n=0$ for some $n\in\mathbb{N}$ then there is a sequence $\{x_n\}$ in $(\frac{\alpha}{n+1},\frac{\alpha}{n}]$  such that $\phi(x_n)\rightarrow0$ as $n\rightarrow\infty$. Then $\exists$ a converging subsequence $\{x_{n_k}\}$ of $\{x_n\}$. Let $$\lim_{k\rightarrow\infty}x_{n_k}=x\in[\frac{\alpha}{n+1},\frac{\alpha}{n}].$$ Then as $\phi$ is lower semi-continuous and $\phi(x)=0$ iff $x=0$, $$\phi(x)\leq\lim_{k\rightarrow\infty}\phi(x_{n_k})=0\Longrightarrow\phi(x)=0,$$ which is a contradiction. Hence $b,b_n,c_n>0~~\forall~n\in\mathbb{N}$.\\
  Now define $\phi_1~:[0,\infty)\rightarrow[0,\infty)$ by\\
  $\phi_1(x)=
  \left\{
	\begin{array}{llll}
		0  & \mbox{if } x=0 \\
		\inf\{b,a_1\} & \mbox{if } x\in(\alpha,2\alpha)\\
        \inf\{b,a_1,b_1,a_2,....,b_n,a_{n+1}\} & \mbox{if } x\in(\frac{\alpha}{n+1},\frac{\alpha}{n}]~~n\in\mathbb{N}\\
        \inf\{b,a_1,c_1,....c_n\} & \mbox{if } x\in[(n+1)\alpha,(n+2)\alpha)~~n\in\mathbb{N}
	\end{array}
\right.
$\\
Then $\phi_1$ is lower semi-continuous and $\phi_1(x)\leq\phi(x)$ and also $\psi(f(x))-\phi(f(x))\leq\psi(g(x))-\phi_1(g(x))$.
\end{proof}
\begin{lem}
\textit{Let $(\textbf{X},d)$ be a metric space and T : $\textbf{X}\rightarrow\textbf{X}$ be a mapping satisfying (\(\psi,\phi\))-weak contraction with $\psi$ is continuous, nondecreasing and $\phi$ is lower semicontinous, then there exists another $\phi_{2}\in\Phi$ with $\phi_{2}(x)\leq\psi(x)$ and $\phi_{2}(x)\leq\phi(x)$ for all $x\in[0,\infty).$}
\end{lem}
\begin{proof}
If $\phi(x)\leq\psi(x),\forall{x\in[0,\infty)}$, then by taking $\phi_{2}=\phi$, it is done. So let, there exists some $x\in[0,\infty)$ such that $\phi(x)>\psi(x)$. As $\phi(0)=\psi(0)=0$ and $\phi$ is lower semi-continuous\\
\textbf{Case I:} $\exists$ an interval $(a,b)$ containing $x$ such that $\phi(y)>\psi(y)~\forall~y\in(a,b)$ and $\phi(z)\leq\psi(z)$ whenever $z=a,b$.\\
\textbf{Case II:} $\exists$ an interval $(a,\infty)$ containing $x$ such that $\phi(y)>\psi(y)~\forall~y\in(a,\infty)$ and $\phi(a)\leq\psi(a)$.\\
Now in case I if for another $x_1\in[0,\infty)$ there is an interval $(c,d)$, then either $a=c,~b=d$ or, $(a,c)$ and $(b,d)$ are disjoint, otherwise $\exists~z\in\{a,b,c,d\}$ and $z\in(a,b)\cup(c,d)$ such that $\phi(z)\leq\psi(z)$ which is a contradiction as $z\in(a,b)\cup(c,d)$ implies $\phi(z)>\psi(z)$.\\
And in case II there is at most one such interval for $\phi,\psi$, otherwise if there is another interval $(e,\infty)$ for some $x_2\in[0,\infty)$ then either $a\in(e,\infty)$ implies $\phi(a)\leq\psi(a)$, a contradiction, or $e\in(a,\infty)$ implies $\phi(e)\leq\psi(e)$, a contradiction, and its disjoint with interval described in  case I. As, if they are not disjoint $\exists~z\in(a,\infty)$ where $\phi(z)\leq\psi(z)$, which is again a contradiction.\\
Now let $A$=\{$x\in[0,\infty)$:there exists intervals containing $x$ described as in case I\},  and\\
~~~~~~~$B$=\{$x\in[0,\infty)$:there exists intervals containing $x$ described as in case II\}.\\
So then for $\psi,\phi$, $\exists$ a countable set $\Lambda_1$ of disjoint intervals $(p,q)$ of type case I such that for each $x\in~A~\exists~(p,q)\in\Lambda_1$ containing it, and a  set $\Lambda_2$ consisting at most one interval $(r,\infty)$ of type case II such that for each $x\in~B,~x\in(r,\infty)\in\Lambda_1$\\
Now define $\phi_{2}~:[0,\infty)\rightarrow[0,\infty)$ as\\
$\phi_{2}(x)=
\left\{
\begin{array}{lll}
		\phi(x)  & \mbox{if } x\in~[0,\infty)-\{x\in[0,\infty):\phi(x)>\psi(x)\}\\
		\psi(x) & \mbox{if } x\in~(p,q)\in~\Lambda_1\\
        \phi(r) & \mbox{if } x\in~(r,\infty)\in~\Lambda_2
	\end{array}
\right.
$\\
Then $\phi_2$ is lower semi-continuous and $\phi_{2}(x)\leq\phi(x)$ as well as $\phi_{2}(x)\leq\psi(x)$ for all $x\in[0,\infty)$. Hence we can assume for ($\psi,\phi$)-weak contraction with $\psi$ is continuous, nondecreasing and $\phi$ is lower semicontinous, $\psi(x)\geq\phi(x)$ and $\phi$ is continuous at 0.
\end{proof}
\begin{thm}
\textit{Let $(\textbf{X},\preceq,G)$ be a complete partially ordered G-metric space and let T : $\textbf{X}\rightarrow\textbf{X}$ be a nondecreasing map such that $x_0\preceq{Tx_0}$ for some $x_0\in\textbf{X}$. Suppose that there exists $\psi\in\Psi$ and $\phi\in\Phi$ such that for all $x\preceq{y}\preceq{z}$ in $\textbf{X}$}
$$\psi(G(Tx,Ty,Tz))\leq\psi(G(x,y,z))-\phi(G(x,y,z)).$$
\textit{Now if either T is G-continuous or $(\textbf{X},\preceq,G)$ is nondecreasing then T has a fixed point in \textbf{X}.}
\end{thm}
\begin{proof}
Since $G(x,y,z)\leq{M(x,y,z)}~~\forall~x,y,z\in{\textbf{X}}$, so by above two lemmas $\exists$ $\phi_{1}\in\Phi$ such that
$$\psi(G(Tx,Ty,Tz))\leq\psi(G(x,y,z))-\phi(G(x,y,z))\leq\psi(M(x,y,z))-\phi_{1}(M(x,y,z)).$$
Therefore, by Theorem 2.1,  T has a fixed point.
\end{proof}
By the following example we will show that Theorem 2.1 is generalization of Theorem 2.5 and if we transform metric to G-metric, it is not equivalent to the corollary 3.3 in \cite{SRZK}.\\
\begin{exmp}
Let $\textbf{X}$=$[0,1]$, and  $x\preceq{y}$ implies $x\geq{y}$ and define metric on \textbf{X} by $d(x,y)=|x-y|$. Then
$G(x,y,z)$= max$\{|x-y|,|y-z|,|z-x|\}$ is G-metric on \textbf{X} and $d(x,y)=G(x,y,y)$.\\
Therefore $(\textbf{X},\preceq,d)$ is complete partially ordered metric space and also $(\textbf{X},\preceq,G)$ is complete partially ordered G-metric space. Let $T:[0,1]\rightarrow[0,1]$ defined by\\
$Tx=\left\{
           \begin{array}{ll}
             2x+\frac{1}{16} & \mbox{if } 0\leq~x\leq\frac{7}{32}\\
             \frac{16}{25}x+\frac{9}{25} & \mbox{if } \frac{7}{32}<x\leq1
           \end{array}
         \right.
         $\\
Then, T is nondecreasing and continuous.\\
Let $\psi(t)=t$ and $\phi(t)=\frac{1}{32}t$\\Let $x\preceq~y\preceq~z$.\\
Then, $G(Tx,Ty,Tz)$= max$\{|Tx-Ty|,|Ty-Tz|,|Tz-Tx|\}=|Tx-Tz|$ as $x\geq~y\geq~z$,\\
and $M(x,y,z)$= max$\{|x-Tx|,|y-Ty|,|z-Tz|,|x-y|,|x-z|,|y-z|,|y-Tx|,|z-Tx|,\frac{1}{2}[G(x,Ty,Tz)+G(Tx,y,z)]\}$.
As, $T$ is monotone increasing, $|Tx-z|\leq~M(x,y,z)\leq|1-z|$\\
Now Let $A=M(x,y,z)-G(Tx,Ty,Tz)\geq~|Tx-z|-|Tx-Tz|\geq0$ as $Tz\geq~z$.\\
Then, $A\geq
\left\{
           \begin{array}{ll}
             z+\frac{1}{16} & \mbox{if } 0\leq~z\leq\frac{7}{32}\\
             \frac{9}{25}(1-z) & \mbox{if } \frac{7}{32}<z\leq1
           \end{array}
         \right.
         $\\
that is $A\geq\frac{1}{32}|1-z|\geq\frac{1}{32}M(x,y,z)=\phi(M(x,y,z))$.\\
Hence, $\psi(M(x,y,z))-\psi(G(Tx,Ty,Tz))=A\geq\phi(M(x,y,z))$  as $\psi(t)=t$.\\
implies $\psi(G(Tx,Ty,Tz)\leq\psi(M(x,y,z))-\phi(M(x,y,z))$.\\
Hence by Theorem 2.1 T has a fixed point. Here T has a fixed point at 1.
Now if $x=\frac{5}{64}$ and $y=z=0$ then, $|Tx-Ty|=|Tx-Tz|=\frac{5}{32}$ ,\\
but $|x-y|=|x-z|=\frac{5}{64}$, $|Tx-x|=\frac{9}{64}$, $|Ty-y|=\frac{1}{16}$,\\ $|Tx-y|=\frac{7}{32}$, $|Ty-x|=\frac{1}{64}$.\\
So $G(x,y,z)=\frac{5}{64},G(Tx,Ty,Tz)=\frac{5}{32}$ and $M(x,y)$= max$\{|x-y|,|Tx-x|,|Ty-y|,\frac{1}{2}(|Tx-y|+|Ty-x|)\}=\frac{9}{64}<|Tx-Ty|$. \\
Hence,for any $\psi,\phi$; $T$ doesn't satisfy the contractive condition of Theorem 2.5 and as well as the contractive condition of corollary 3.3 in \cite{SRZK}.
\end{exmp}

 Now the following theorems can be proved in similar way as Theorem 2.1 and Theorem 2.5 .
\begin{thm}
\textit{Let $(\textbf{X},\preceq,G)$ be a complete partially ordered G-metric space and let T : $\textbf{X}\rightarrow\textbf{X}$ be a nondecreasing map such that $x_0\preceq{Tx_0}$ for some $x_0\in\textbf{X}$. Suppose that there exists $\psi\in\Psi$ and $\phi\in\Phi$ such that for all comparable $x,y\in\textbf{X}$ with $x\preceq{y}$,}
$$\psi(G(Tx,Ty,T^2x))\leq\psi(N(x,y,Tx))-\phi(N(x,y,Tx)),\eqno{(2.23)}$$
\textit{where M(x,y,Tx)= max $\{G(x,Tx,y),~G(Tx,T^2x,T^2x),~\frac{1}{2}[G(x,Tx,Tx)+G(y,Ty,Ty)],\\
~~~~~~~~~~~~~~~~~~~~~~~~~~~~~~~~~~~~~~~~~~~~~~~~~~~~~~~~~~~~~~~~~~~~~\frac{1}{2}[G(x,T^2x,Ty)+G(Tx,Tx,y)]\}$}.\\
\textit{Now if either T is G-continuous or $(\textbf{X},\preceq,G)$ is nondecreasing then T has a fixed point in \textbf{X}.}
\end{thm}
\begin{thm}
\textit{Let $(\textbf{X},\preceq,G)$ be a complete partially ordered G-metric space and let T : $\textbf{X}\rightarrow\textbf{X}$ be a nondecreasing map such that $x_0\preceq{Tx_0}$ for some $x_0\in\textbf{X}$. Suppose that there exists $\psi\in\Psi$ and $\phi\in\Phi$ such that for all comparable $x,y\in\textbf{X}$ with $x\preceq{y}$}
$$\psi(G(Tx,T^{2}x,Ty))\leq\psi(G(x,Tx,y))-\phi(G(x,Tx,y)).$$
\textit{Now if either T is G-continuous or $(\textbf{X},\preceq,G)$ is nondecreasing then T has a fixed point in \textbf{X}.}
\end{thm}
The following example shows that Theorem 2.6 is generalization  of Theorem 2.7 .
\begin{exmp}
Let $\textbf{X}=\{1,2,3\}$ and $x\preceq{y}~~\textmd{if  }~x\leq{y}$ and define $G$-metric on $\textbf{X}$ by \\
$G(1,1,1)=G(2,2,2)=G(3,3,3)=0,~G(1,1,2)=G(2,2,3)=3,~G(1,2,2)=G(1,2,3)=5,\\~G(1,1,3)=G(2,3,3)=4,~G(1,3,3)=2$.\\
Then $(\textbf{X},\preceq,G)$ be a complete partially ordered $G$-metric space.\\
 Now Define $T:\textbf{X}\rightarrow\textbf{X}$ by $T(1)=2,~~T(2)=T(3)=3$ and let
$\psi(t)=t$ and $\phi(t)=\frac{t}{20}$.\\
Then $T$ is $G$-continuous, non-decreasing and $1\preceq~T(1)=2$. Now we show that T satisfies the contractive condition in Theorem 2.6 .\\
Now if$~~~~~x=1,y=1~~~G(Tx,Ty,T^2x)=3,~~~G(x,y,Tx)=3,~~~N(x,y,Tx)=5\Longrightarrow\phi(N(x,y,Tx)=\frac{1}{4}$\\\\
$~~~~~~~~~~~~~~~x=1,y=2~~~G(Tx,Ty,T^2x)=4,~~~G(x,y,Tx)=5,~~~N(x,y,Tx)=5\Longrightarrow\phi(N(x,y,Tx)=\frac{1}{4}$\\\\
$~~~~~~~~~~~~~~~x=1,y=3~~~G(Tx,Ty,T^2x)=4,~~~G(x,y,Tx)=5,~~~N(x,y,Tx)=5\Longrightarrow\phi(N(x,y,Tx)=\frac{1}{4}$\\\\
$~~~~~~~~~~~~~~~x=2,y=2~~~G(Tx,Ty,T^2x)=0,~~~G(x,y,Tx)=3,~~~N(x,y,Tx)=4\Longrightarrow\phi(N(x,y,Tx)=\frac{1}{5}$\\\\
$~~~~~~~~~~~~~~~x=3,y=3~~~G(Tx,Ty,T^2x)=0,~~~G(x,y,Tx)=0,~~~N(x,y,Tx)=0\Longrightarrow\phi(N(x,y,Tx)=0$.\\\\
Hence $\psi(G(Tx,Ty,T^2x)\leq\psi(N(x,y,Tx))-\phi(N(x,y,Tx))~~~~\forall~x\preceq{y}$. Here, T has a fixed point at 3.\\
But $\nexists$ any $\psi\in\Psi$ and $\phi\in\Phi$ such that\\
$\psi(G(Tx,Ty,T^2x)\leq\psi(G(x,y,Tx))-\phi(G(x,y,Tx)),~~~~\forall~x\preceq{y}$ holds.
\end{exmp}
\section{Uniqueness of fixed point}
The following example shows that conditions of the Theorem 2.1 are not sufficient for the uniqueness of fixed point.
\begin{exmp}
Let $\textbf{X}=\{2,3\}$ and $x\preceq~y$ iff $x\mid{y}$ and defined G-metric by $G(x,y,z)$=max $\{|x-y|,|y-z|,|z-x|\}$. Then $(\textbf{X},\preceq,G)$ is complete partially ordered G-metric space and 2 is not comparable with 3. Consider the mapping $T:\textbf{X}\rightarrow\textbf{X}$ by T(2)=2 and T(3)=3. Then $G(Tx,Ty,Tz)=0~~~~\forall~x\preceq{y}\preceq{z}$. Therefore taking $\psi(t)=t$ and $\phi(t)=\frac{1}{2}t$, $(\textbf{X},\preceq,G)$ and T satisfies all conditions in Theorem 2.1. However, T has two fixed point in $\textbf{X}$.

\end{exmp}
In the next theorem we give a sufficient conditions for the uniqueness of the fixed point.
\begin{thm}
\textit{Suppose all the condition in Theorem 2.1 holds and let, for any  $x,y\in{T_\textbf{X}}$ there exists $z\in\textbf{X}$ such that, $x\preceq{z}$ and $y\preceq{z}$ and also  $\{T^mz\}$ is a convergent sequence in $(\textbf{X},G)$. Then T has unique fixed point. Where $T_\textbf{X}$ is set of all fixed point of T in $\textbf{X}$}.
\end{thm}
\begin{proof}
 Let T has two fixed point $x$ and $y$ in $\textbf{X}$. Consider the following two cases.\\
$~~~1.$Now If $x$ and $y$ are comparable. Without loss of generality let $x\preceq{y}$ then by $(2.1)$ we have
$$\psi(G(x,x,y))=\psi(G(Tx,Tx,Ty))\leq\psi(M(x,x,y))-\phi(M(x,x,y)),\eqno{(3.1)}$$
where
$M(x,x,y)$= max $\{G(x,Tx,x),~G(x,Tx,y),~G(x,x,y),~G(x,Tx,Tx),\\
~~~~~~~~~~~~~~~~~~~~~~~~~~~~~~~~~~~~~~~~~~~~~~~~~~~~~~~~G(y,Ty,Ty),~\frac{1}{2}[G(x,Tx,Ty)+G(Tx,x,y)]\}$ \\
$~~~~~~~~~~~~~~~~~~~~~~=~G(x,x,y).$\\
Therefore, from  $(3.1)$ we get
$$\psi(G(x,x,y))\leq\psi(G(x,x,y))-\phi(G(x,x,y)),$$
which lead us to a contradiction unless $G(x,x,y)=0$. Thus $x=y$.\\
$~~~2.$ If $x$ and $y$ are not comparable. Then by assumption in theorem, there is a $z\in\textbf{X}$ such that $x\preceq{z}$ and $y\preceq{z}$. Since $T$ is non-decreasing, $T^mx=x\preceq~T^mz$ and $T^my=y\preceq{T^mz}$ for all $m=0,1,2,3,.........$\\
Then,
 $$\psi(G(x,x,T^{m+1}z))=\psi(G(Tx,Tx,T^{m+1}z))\leq\psi(M(x,x,T^mz))-\phi(M(x,x,T^mz)),\eqno{(3.2)}$$\\
 where, \\$M(x,x,T^mz)$= max $\{G(x,Tx,x),~G(x,Tx,T^mz),~G(x,x,T^mz),~G(x,Tx,Tx),\\
~~~~~~~~~~~~~~~~~~~~~~~~~~~~~~G(T^mz,T^{m+1}z,T^{m+1}z),~\frac{1}{2}[G(x,Tx,T^{m+1}z)+G(Tx,x,T^mz)]\}$ \\
$~~~~~~~~~~~~~~~~~~~~~~~~~~~~$= max $\{G(x,x,T^mz),~G(T^mz,T^{m+1}z,T^{m+1}z),~\frac{1}{2}[G(x,x,T^{m+1}z)+G(x,x,T^mz)]\}~~~~~(3.3)$ \\
  $$\textmd{Now since, $\{T^mz\}$ is convergent,}\lim_{m\rightarrow\infty}G(T^mz,T^{m+1}z,T^{m+1}z)=0.~~~~~~~~~~~~~~~~~~~~~~~~~~~~~~~~~~~~~~~~~~~~~~~~~~$$
Then, \textit{Case I: } $\exists~~~K\in\mathbb{N}$ such that for all $m>K$,\\
 $G(T^mz,T^{m+1}z,T^{m+1}z)\leq${ max $\{G(x,x,T^mz),~\frac{1}{2}[G(x,x,T^{m+1}z)+G(x,x,T^mz)]\}$}. \\
So then, $M(x,x,T^mz)$= max $\{G(x,x,T^mz),~\frac{1}{2}[G(x,x,T^{m+1}z)+G(x,x,T^mz)]\}$. \\
Now if $G(x,x,T^{m+1}z)>G(x,x,T^mz)$, then $\frac{1}{2}[G(x,x,T^{m+1}z)+G(x,x,T^mz)]<G(x,x,T^{m+1}z)$ and $M(x,x,T^mz)=~\frac{1}{2}[G(x,x,T^{m+1}z)+G(x,x,T^mz)]>0\Longrightarrow\phi(G(~\frac{1}{2}[G(x,x,T^{m+1}z)+G(x,x,T^mz)])>0$ then from (3.3) we get \\
  $\psi(G(x,x,T^{m+1}z))\leq\psi(\frac{1}{2}[G(x,x,T^{m+1}z)+G(x,x,T^mz)])-\phi(\frac{1}{2}[G(x,x,T^{m+1}z)+G(x,x,T^mz)])$\\
  $~~~~~~~~~~~~~~~~~~~~~~~~~<\psi(\frac{1}{2}[G(x,x,T^{m+1}z)+G(x,x,T^mz)])\leq\psi(G(x,x,T^{m+1}z))$,
  which is a contradiction. So, we have $G(x,x,T^{m+1}z)\leq~M(x,x,T^mz)=G(x,x,T^mz)$\\
  Hence,$\{G(x,x,T^mz)\}$ is a positive, decreasing sequence in $\mathbb{R}$ for $m>K$, which is bounded below, so it is convergent and therefore $\exists$   $a\geq0$ such that
  $$\lim_{m\rightarrow\infty}G(x,x,T^mz)=a.~~~~~~~~~~~~~~~~~~~~~~~~~~~~~~~~~~~~~~~~~~~~~~~~~~~~~~~~~~~~~~~~~~~~~~~~~~~~~~~~~~~~~~~~~~~~~~~~~~~~~~~~$$
  $$\textmd{Now if, $a>0$ then $\phi(a)>0$, letting limit $m\rightarrow\infty$ in $(3.2)$, we get~~~~~}\psi(a)\leq\psi(a)-\phi(a).~~~~~~~~~~~~~~~~~~~~$$
  Which is a contradiction. So $a=0$ that is
  $$\lim_{m\rightarrow\infty}G(x,x,T^mz)=0.~~~~~~~~~~~~~~~~~~~~~~~~~~~~~~~~~~~~~~~~~~~~~~~~~~~~~~~~~~~~~~~~~~~~~~~~~~~~~~~~~~~~~~~~~~~~~~~~~~~~~~~~~~~~\eqno{(3.4)}$$
\textit{Case II: } There is no $K\in\mathbb{N}$ such that \\ $G(T^mz,T^{m+1}z,T^{m+1}z)\leq{\textmd{max }\{G(x,x,T^mz),\frac{1}{2}[G(x,x,T^{m+1}z)+G(x,x,T^mz)]\}}~~\forall~m>K$ holds.
    $$\textmd{So as }\lim_{m\rightarrow\infty}G(T^mz,T^mz,T^{m+1}z)=0\Longrightarrow\lim_{m\rightarrow\infty}M(x,x,T^mz)=0,~~~~~~~~~~~~~~~~~~~~~~~~~~~~~~~~~~~~~~~~~~~~~~~~~$$
   $$\textmd{then, }\lim_{m\rightarrow\infty}\psi(G(x,x,T^{m+1}z))=0\Longrightarrow\lim_{m\rightarrow\infty}G(x,x,T^{m+1}z)=0.~~~~~~~~~~~~~~~~~~~~~~~~~~~~~~~~~~~~~~~~~~~~~~~~~~~~~~~~~~~~~~~~~~~~~~~~~~~$$
  Now Since $G(x,T^mz,T^mz)\leq~2G(x,x,T^mz)$
  $$\lim_{m\rightarrow\infty}G(x,T^mz,T^mz)=0.~~~~~~~~~~~~~~~~~~~~~~~~~~~~~~~~~~~~~~~~~~~~~~~~~~~~~~~~~~~~~~~~~~~~~~~~~~~~~~~~~~~~~~~~~~~~~~~~~~~~~~~~~~~~\eqno{(3.5)}$$
  $$\textmd{Similarly ,}\lim_{m\rightarrow\infty}G(y,T^mz,T^mz)=0,~~~~~~~~~~~~~~~~~~~~~~~~~~~~~~~~~~~~~~~~~~~~~~~~~~~~~~~~~~~~~~~~~~~~~~~~~~~~~~~~~~~~~~~~~~~~~~~~~~~~~~~~~~~~\eqno{(3.6)}$$
  as $G(x,x,y)\leq{G(x,x,T^mz)+G(y,T^mz,T^mz)}$, so $G(x,x,y)=0$ and hence $x=y$. So uniqueness of fixed point is proved.
  \end{proof}
By the following example we will show that if we remove convergence condition of $\{T^mz\}$ in Theorem 3.1, then T may not have unique fixed point.
\begin{exmp}
Let $\textbf{X}=~\{2,3,12,-18,30,-42,......\}$ i.e, $\textbf{X}$ consists of 2,3 and $(-1)^np_n.6~~\forall~n\in\mathbb{N}$ where $p_1,p_2,.....$ are prime number in usual order with $p_1=2, p_2=3$ and so on.\\
Define $\preceq$ on $\textbf{X}$ by $a\preceq~b$ if $a\mid{b}$. Then $2\preceq(-1)^np_n.6$ and $3\preceq(-1)^np_n.6$.\\
Now since for $n\neq{m}~~p_n\nmid{p_m}$, so $(-1)^np_n.6$ and $(-1)^mp_m.6$ are not comparable and also 2,3 aren't comparable.\\
Define G-metric on \textbf{X} by
$G(x,y,z)$= max$\{|x-y|,|y-z|,|z-x|\}$.
Then ($\textbf{X},\preceq$,G) is complete partially ordered $G$-metric space. Let $T:\textbf{X}\rightarrow\textbf{X}$ defined by\\
$T(2)=2,~T(3)=3,~T((-1)^np_n.6)=(-1)^{n+1}p_{n+1}.6~~\forall~n\in\mathbb{N}$, and $\psi,\phi:[0,\infty)\rightarrow[0,\infty)$ defined by \\
$\psi(t)=t~~~\forall~t\in[0,\infty)$ and $\phi(t)=\left\{
\begin{array}{ll}
\frac{1}{2}t & \mbox{if } x\in[0,2]\\
1 & \mbox{if } x\in(2,\infty)
\end{array}
\right.$\\
Then $T$ is $G$-continuous, nondecreasing, and since $2\mid2\Longrightarrow2\preceq~T(2)$. Now we show that T satisfies the contractive condition in Theorem 2.1. Since only 2 and 3 are comparable with $(-1)^np_n.6$.
Let $x=2,y=3$ and $z=(-1)^np_n.6$ for any $n\in\mathbb{N}$, then $x\preceq{z}$,$y\preceq{z}$ and $\{T^mz\}$ doesn't converge in $\textbf{X}$. \\
So, $G(Tx,Tz,Tz)=G(Tx,Tx,Tz)=G(2,2,(-1)^np_n.6)=|(-1)^np_n.6-2|$ and\\
$M(x,x,z)\\~~~~~~~~~~~=M(2,2,(-1)^np_n.6)$ = max $\{G(2,2,(-1)^np_n.6),G((-1)^np_n.6,(-1)^{n+1}p_{n+1}.6,(-1)^{n+1}p_{n+1}.6),$\\
$~~~~~~~~~~~~~~~~~~~~~~~~~~~~~~~~~~~~~~~~~~~~~~~~~~~~~~~~~~~~\frac{1}{2}[G(2,2,(-1)^{n+1}p_{n+1}.6)+G(2,2,(-1)^np_n.6)]\}$\\
$~~~~~~~~~~~~~~~~~~~~~~~~~~~~~~~~~~~~~~~~=(p_n+p_{n+1}).6$ .\\
Then $\psi(G(Tx,Tx,Tz))=|(-1)^np_n.6-2|\leq{p_{n+1}.6+2}\leq(p_n+p_{n+1}).6-1=\psi(M(x,x,z))-\phi(M(x,x,z))$, as $p_n.6\geq12$.\\
similarly $\psi(G(Ty,Ty,Tz)=|(-1)^np_n.6-3|\leq(p_n+p_{n+1}).6-1=\psi(M(y,y,z))-\phi(M(y,y,z))$.\\
So by Theorem 2.1 T has a fixed point. Here T has two fixed point 2 and 3, so fixed point is not unique, as $\{T^n(-1)^np_n.6\}$ doesn't converge in $\textbf{X}$.
\end{exmp}

\textbf{Remark:}
 Under the conditions of uniqueness of fixed point in the previous theorem, it can be proved by similar way that the Theorem 2.6 has a unique fixed point.\\\\\\
 \textbf{Acknowledgements}\\$~~~$ SMH and SB gratefully acknowledge the support of Aliah University, Kolkata for providing all the facilities when the manuscript was prepared. SB also Acknowledge the financial support of CSIR, Govt. of India.\\

\end{document}